\documentclass[a4paper,11pt,twoside,notitlepage]{amsart}

\usepackage{amsmath,amssymb,amsthm,mathtools}
\usepackage{xcolor}

\definecolor{mutedgray}{HTML}{667085}
\usepackage[ocgcolorlinks,linkcolor=blue,citecolor=blue,urlcolor=blue]{hyperref}
\usepackage{url}
\numberwithin{equation}{section}
\mathtoolsset{showonlyrefs}
\allowdisplaybreaks

\newtheorem{theorem}{Theorem}[section]
\newtheorem{proposition}[theorem]{Proposition}
\newtheorem{lemma}[theorem]{Lemma}
\newtheorem{corollary}[theorem]{Corollary}
\theoremstyle{remark}
\newtheorem{remark}[theorem]{Remark}

\DeclareMathOperator{\Diff}{Diff}
\DeclareMathOperator{\Id}{Id}
\newcommand{\p}{\partial}
\newcommand{\ol}{\overline}
\newcommand{\para}[1]{\vspace{3mm}\noindent\textbf{#1.}}

\title[Rigidity for the weighted $p$-Laplace equation]{A rigidity result for the weighted $p$-Laplace equation in higher dimensions}

\author[C.-L. Lin]{Ching-Lung Lin}
\address{Department of Mathematics, National Cheng Kung University, Tainan 701, Taiwan.}
\email{cllin2@mail.ncku.edu.tw}

\author[Y.-H. Lin]{Yi-Hsuan Lin}
\address{Department of Applied Mathematics, National Yang Ming Chiao Tung University, Hsinchu, Taiwan \& Fakult\"at f\"ur Mathematik, University of Duisburg-Essen, Essen, Germany}
\email{yihsuanlin3@gmail.com}


\hypersetup{
	pdftitle={Rigidity for the weighted p-Laplace equation in higher dimensions},
	pdfauthor={Yi-Hsuan Lin},
	pdfsubject={Rigidity of weights from the linearized boundary response at an affine datum},
	pdfkeywords={Weighted p-Laplace equation, inverse boundary value problem, Dirichlet-to-Neumann map, anisotropic Calderon problem, rigidity}
}

\begin{document}

\begin{abstract}
We study rigidity for the weighted $p$-Laplace equation on smooth bounded domains in dimensions three and higher. We show that a smooth positive weight sufficiently close to $1$ is identically $1$ if the derivative of its Dirichlet-to-Neumann map at one affine boundary value agrees with that for the unit weight. We also prove local rigidity at every smooth positive reference weight which is constant in one direction. The reference weight need not be close to $1$, and the unknown weight is only required to be sufficiently close to it, without a directional restriction. No convexity, analyticity, or agreement near the boundary is assumed. The proof combines linearization at a solution without critical points with recent rigidity results for the anisotropic Calder\'on problem. An identity for the linearized conductivity tensors removes the remaining diffeomorphism. This last step does not require smallness of either weight.
\end{abstract}

\maketitle

\section{Introduction}
\label{sec:introduction}

Let $\Omega\subset\mathbb R^n$ be a smooth bounded connected domain, let $p\in(1,\infty)\setminus\{2\}$, and let $\gamma\in C^\infty(\ol\Omega)$ be strictly positive. We consider the weighted $p$-Laplace equation
\begin{equation}\label{eq:weighted-p-laplace}
\begin{cases}
\operatorname{div}\big(\gamma|\nabla u|^{p-2}\nabla u\big)=0 & \text{in }\Omega,\\
u=f & \text{on }\p\Omega.
\end{cases}
\end{equation}
For each smooth boundary value $f$, there is a unique weak solution $u_f\in W^{1,p}(\Omega)$; see \cite[Theorem~1.1]{CarsteaFeizmohammadi2025}. It minimizes $p^{-1}\int_\Omega\gamma|\nabla u|^p\,dx$ among functions with trace $f$. The Dirichlet-to-Neumann (DN) map is defined weakly by
\begin{equation}\label{eq:nonlinear-DN}
\langle\Lambda_\gamma(f),h\rangle=\int_\Omega\gamma|\nabla u_f|^{p-2}\nabla u_f\cdot\nabla H\,dx,
\end{equation}
where $H$ is a smooth extension of $h$ to $\ol\Omega$. The weak equation shows that this definition is independent of the extension. For a smooth solution, $\Lambda_\gamma(f)=\gamma|\nabla u_f|^{p-2}\p_\nu u_f|_{\p\Omega}$, where $\nu$ is the Euclidean outward unit normal. We aim to determine $\gamma$ from these boundary measurements.

For $p=2$, the problem reduces to the scalar conductivity problem introduced in \cite{Calderon1980}. Global uniqueness for smooth scalar conductivities in dimensions $n\ge3$ was proved in \cite{SylvesterUhlmann1987}. We refer to \cite{FeldmanSaloUhlmann2025} for the linear problem, its anisotropic formulation and further references.

For the weighted $p$-Laplace equation, the DN map determines the boundary value of the weight \cite{SaloZhong2012} and its normal derivative \cite{Brander2016}. Enclosure and monotonicity methods have been used to detect inclusions and recover their convex hulls \cite{BranderKarSalo2015,BranderHarrachKarSalo2018}. Under a pointwise ordering assumption, uniqueness for the weight was proved in the plane and, for weights close to a constant, in higher dimensions \cite{GuoKarSalo2016}. Uniqueness for arbitrary smooth positive weights in the plane was later obtained in \cite[Theorem~1.2]{CarsteaFeizmohammadi2025}. The higher-dimensional theorem in \cite[Theorem~1.3]{CarsteaFeizmohammadi2025} assumes real analytic weights, smallness of a directional derivative, and a flat boundary patch. Smooth uniqueness is also known when both weights are constant in a common direction \cite[Corollary~1.1]{CarsteaFeizmohammadi2025}.

Higher-order linearization has been used in many nonlinear inverse problems. For semilinear elliptic equations, differentiating the DN map gives integral identities involving products of solutions of a linear equation \cite{FeizmohammadiOksanen2020,LassasLiimatainenLinSalo2021}. These identities contain information which is absent from the first linearization. The method also applies to partial data problems and simultaneous recovery of coefficients and unknown boundary components \cite{LassasLiimatainenLinSalo2021Partial}. For isotropic quasilinear conductivities, higher derivatives of the DN map determine the differentials of the conductivity by a completeness property for products of solutions \cite{CarsteaFeizmohammadiKianKrupchykUhlmann2021}. A density result for tensor products of at least three gradients of harmonic functions was proved in \cite{CarsteaFeizmohammadi2023}.

For \eqref{eq:weighted-p-laplace}, the linearization at a constant solution is degenerate when $p>2$, and the flux is not differentiable at zero gradient when $1<p<2$. We instead use a solution $u_0$ with no critical points. At this solution, the first derivative of the nonlinear DN map is the conductivity DN map for $A=\gamma DJ(\nabla u_0)$, where $J(a)=|a|^{p-2}a$; see \cite[Proposition~2.3]{CarsteaFeizmohammadi2025}. The matrix is anisotropic and depends on both the unknown weight and its background solution. The higher-order results mentioned above do not directly apply to this linearized equation.

For anisotropic conductivities, a change of coordinates fixing the boundary leaves the DN map unchanged. In dimensions $n\ge3$, the conductivity equation can be written as a Laplace--Beltrami equation, and one seeks to determine the metric up to this invariance. Uniqueness for real analytic metrics was established in \cite{LeeUhlmann1989,LassasUhlmann2001}; a proof by the Poisson embedding method is given in \cite{LassasLiimatainenSalo2020}. Results within a fixed conformal class are known in admissible and transversally anisotropic geometries \cite{DosSantosFerreiraKenigSaloUhlmann2009,DosSantosFerreiraKurylevLassasSalo2016}. More recently, rigidity relative to the Euclidean metric was proved for smooth metrics sufficiently close to it \cite{Lin2026NearEuclidean,Stefanov2026NearEuclidean}. These results require neither convexity of the domain nor a transversal product structure. Theorem~2 of \cite{Stefanov2026NearEuclidean} also allows a conformally Euclidean reference metric whose conformal factor is close to $1$.

In this work, we apply these rigidity results to the weighted $p$-Laplace equation. For the unit weight and affine boundary data, the solution is affine, and its linearized conductivity matrix is constant. A fixed linear change of variables makes this matrix the identity. By anisotropic rigidity, the unknown linearized conductivity and this constant matrix are related by a diffeomorphism fixing the boundary. We show that the form $\gamma DJ(\nabla u_0)$ forces the diffeomorphism to be the identity. The proof uses only the first derivative with respect to the boundary value. There is no ordering assumption on the weights.

\para{Main results}
Fix $\zeta\in\mathbb S^{n-1}$ and set $f_0=(\zeta\cdot x)|_{\p\Omega}$. We write $D_f\Lambda_\gamma(f_0)$ for the first Fr\'echet derivative of $f\mapsto\Lambda_\gamma(f)$ at $f_0$, with the weight $\gamma$ fixed. For a smooth boundary function $h$, its action is $D_f\Lambda_\gamma(f_0)[h]=\left.\frac{d}{dt}\right|_{t=0}\Lambda_\gamma(f_0+th)$. An equality $D_f\Lambda_{\gamma_1}(f_0)=D_f\Lambda_{\gamma_2}(f_0)$ means that $D_f\Lambda_{\gamma_1}(f_0)[h]=D_f\Lambda_{\gamma_2}(f_0)[h]$ as functions on $\p\Omega$ for every $h\in C^\infty(\p\Omega)$. Lemma~\ref{lem:noncritical-linearization} proves the differentiability in the function spaces used below.

We normalize the constant reference weight to $1$. Multiplying a weight by a positive constant does not change the Dirichlet solution and multiplies the DN map and its boundary derivative by the same constant. This normalization involves no loss of generality.

\begin{theorem}[Rigidity at the unit weight]\label{thm:constant-holder}
Let $n\ge3$, set $k=\lfloor n/2\rfloor+1$, and suppose that $0<\alpha<1$ and $k+\alpha>n/2+1$. There is $\varepsilon>0$, depending only on $\Omega,n,p,k,\alpha$ and $\zeta$, such that any positive $\gamma\in C^\infty(\ol\Omega)$ satisfying $\|\gamma-1\|_{C^{k,\alpha}(\ol\Omega)}<\varepsilon$ and
\begin{equation}\label{eq:constant-linearized-data}
D_f\Lambda_\gamma(f_0)=D_f\Lambda_1(f_0)
\end{equation}
must satisfy $\gamma=1$ in $\Omega$.
\end{theorem}

The weight is smooth, but the smallness condition uses only a finite H\"older norm. For $n=3$, this is $C^{2,\alpha}$ with $\alpha>1/2$. We also have the following Sobolev version.

\begin{corollary}[Sobolev smallness]\label{cor:constant-sobolev}
Let $n\ge3$ and let $s>n/2+1$ be an integer. There is $\varepsilon>0$, depending only on $\Omega,n,p,s$ and $\zeta$, such that any smooth positive weight satisfying $\|\gamma-1\|_{H^s(\Omega)}<\varepsilon$ and \eqref{eq:constant-linearized-data} must equal $1$.
\end{corollary}

We next consider an arbitrary fixed smooth positive reference weight which is constant in one direction. No closeness to $1$ is assumed for this reference weight.

\begin{theorem}[Rigidity at a reference weight constant in one direction]\label{thm:directional-reference}
	Let $n\ge3$ and let $s>n/2+1$ be an integer. Fix $\gamma_0\in C^\infty(\ol\Omega)$ with $\gamma_0>0$ on $\ol\Omega$ and $\zeta\cdot\nabla\gamma_0=0$ in $\Omega$. There is $\varepsilon>0$, depending on $\Omega,n,p,s,\zeta$ and $\gamma_0$, such that any positive $\gamma\in C^\infty(\ol\Omega)$ satisfying
	\begin{equation}\label{eq:directional-assumptions}
		\|\gamma-\gamma_0\|_{H^s(\Omega)}<\varepsilon
	\end{equation}
	and $D_f\Lambda_\gamma(f_0)=D_f\Lambda_{\gamma_0}(f_0)$ must satisfy $\gamma=\gamma_0$ in $\Omega$.
\end{theorem}

For $\zeta=e_1$, the reference weight can be a sufficiently small smooth perturbation of $1$ depending only on $x_2,\ldots,x_n$. The unknown weight may depend on all variables. In \cite[Corollary~1.1]{CarsteaFeizmohammadi2025}, both weights are constant in the same direction, with no smallness assumption. Here we remove the directional restriction on the unknown weight under \eqref{eq:directional-assumptions}.

Equality of the nonlinear DN maps in a neighborhood of $f_0$ gives the derivative equality in each of these results. In particular, equality of the full nonlinear DN maps suffices. We assume no agreement near $\p\Omega$, analyticity or convexity of $\Omega$. The results concern the specified reference weights, not an arbitrary pair of weights close to $1$.

The step which removes the diffeomorphism does not use a perturbation argument. By $DJ(a)a=(p-1)J(a)$, the nonlinear background solution is also harmonic for its linearized conductivity. A diffeomorphism between the conductivities must relate the background solutions. Substituting this relation into the conductivity transformation law gives $DF\,DF^T=dI$ for a positive function $d$. Using the boundary condition in an integral inequality, we obtain $DF=I$. We prove this in Proposition~\ref{prop:eliminate-diffeomorphism} without a smallness assumption on either weight.

\section{Noncritical solutions and the linearized DN map}
\label{sec:linearization}

For $a\in\mathbb R^n\setminus\{0\}$, set $J(a)=|a|^{p-2}a$ and
\begin{equation}\label{eq:linearization-matrix}
B(a)=DJ(a)=|a|^{p-2}\Big(I+(p-2)\frac{a\otimes a}{|a|^2}\Big).
\end{equation}
The eigenvalues of $B(a)$ are $|a|^{p-2}$ on $a^\perp$ and $(p-1)|a|^{p-2}$ in the direction of $a$. In particular, $B(a)$ is positive definite for $p>1$.

For a smooth positive definite symmetric matrix field $A$, let $\mathcal D_A: H^{1/2}(\p\Omega) \to H^{-1/2}(\p\Omega)$ be the DN map for $\operatorname{div}(A\nabla v)=0$. With Euclidean surface measure on the boundary, it is defined by $\langle\mathcal D_Ah,\psi\rangle=\int_\Omega A\nabla v_h\cdot\nabla H\,dx$, where $v_h$ is $A$-harmonic with trace $h$ and $H$ has trace $\psi$. For smooth data, $\mathcal D_Ah=\nu\cdot A\nabla v_h|_{\p\Omega}$.

We use the same boundary value $f_0=(\zeta\cdot x)|_{\p\Omega}$ for every weight sufficiently close to $1$. The next lemma constructs the corresponding solution as a small $C^1$ perturbation of $\zeta\cdot x$ and proves that its gradient stays away from zero. We also give the smooth dependence on nearby boundary values and the noncritical linearization of \cite[Propositions~2.2 and~2.3]{CarsteaFeizmohammadi2025} in the norms needed below.

\begin{lemma}\label{lem:noncritical-linearization}
Let $X=C^{k,\alpha}(\ol\Omega)$, $Y=C^{k+1,\alpha}(\ol\Omega)$ and $\mathcal B=C^{k+1,\alpha}(\p\Omega)$, with $k,\alpha$ as in Theorem~\ref{thm:constant-holder}. Alternatively, take $X=H^s(\Omega)$, $Y=H^{s+1}(\Omega)$ and $\mathcal B=H^{s+1/2}(\p\Omega)$, with $s$ as in Corollary~\ref{cor:constant-sobolev}. There are neighborhoods of $1$ in $X$ and of $f_0$ in $\mathcal B$ on which the solution map $(\gamma,f)\mapsto u_{\gamma,f}\in Y$ is smooth and $|\nabla u_{\gamma,f}|\ge1/2$ on $\ol\Omega$. Set $u_\gamma=u_{\gamma,f_0}$ and $A_\gamma=\gamma B(\nabla u_\gamma)$. On a fixed sufficiently small ball about $1$, we have
\begin{equation}\label{eq:background-estimates}
\|u_{\gamma_1}-u_{\gamma_2}\|_Y+\|A_{\gamma_1}-A_{\gamma_2}\|_X\le C\|\gamma_1-\gamma_2\|_X
\end{equation}
with a constant independent of the two weights. The nonlinear DN map is smooth with values in $C^{k,\alpha}(\p\Omega)$ in the H\"older case and $H^{s-1/2}(\p\Omega)$ in the Sobolev case, and
\begin{equation}\label{eq:linearized-DN}
D_f\Lambda_\gamma(f_0)=\mathcal D_{A_\gamma}.
\end{equation}
For smooth weights and smooth boundary data in these neighborhoods, the solution is smooth up to the boundary.
\end{lemma}

\begin{proof}
Let $Y_0=\{w\in Y:w|_{\p\Omega}=0\}$, and choose a bounded right inverse $E:\mathcal B\to Y$ of the trace map. Put $Z=C^{k-1,\alpha}(\ol\Omega)$ in the H\"older case and $Z=H^{s-1}(\Omega)$ in the Sobolev case. With $f=f_0+h$ and $u=\zeta\cdot x+Eh+w$, define
\begin{equation}\label{eq:forward-map}
\mathcal F(\gamma,h,w)=\operatorname{div}\big(\gamma J(\zeta+\nabla Eh+\nabla w)\big).
\end{equation}
Both choices of $X$ are Banach algebras. The maps $\nabla:Y\to X^n$ and $\operatorname{div}:X^n\to Z$ are bounded, and $Y\hookrightarrow C^1(\ol\Omega)$. We choose $(h,w)$ so that $\|\nabla Eh+\nabla w\|_{L^\infty}<1/2$ and restrict $\gamma$ to a positive neighborhood of $1$. The argument of $J$ stays in a fixed compact annulus away from zero. Smooth composition and multiplication in $X$ make $\mathcal F$ a smooth map to $Z$. At $(1,0,0)$, we have $\mathcal F=0$ and $D_w\mathcal F=L_0:=\operatorname{div}(B(\zeta)\nabla\cdot)$. The constant coefficient Dirichlet operator $L_0:Y_0\to Z$ is an isomorphism.

The implicit function theorem gives a smooth branch $w=w(\gamma,h)$ with $w(1,0)=0$. Since $E:\mathcal B\to Y$ is bounded and $Y\hookrightarrow C^1(\ol\Omega)$, we can choose neighborhoods of $1$ in $X$ and of $0$ in $\mathcal B$ such that $\|\nabla Eh+\nabla w(\gamma,h)\|_{L^\infty}<1/2$ throughout their product. The function $u_{\gamma,f_0+h}=\zeta\cdot x+Eh+w(\gamma,h)$ solves \eqref{eq:weighted-p-laplace} with boundary value $f_0+h$. Using $|\zeta|=1$, we obtain 
\[
|\nabla u_{\gamma,f_0+h}(x)|\ge1-\|\nabla Eh+\nabla w(\gamma,h)\|_{L^\infty}>1/2 \quad \text{for every }x\in\ol\Omega.
\] 
In particular, the fixed boundary value $f_0$ gives a solution without critical points for every weight in the chosen neighborhood of $1$. The gradient also has a uniform upper bound, and the weights are uniformly positive, so the linearized matrices are uniformly elliptic. The constructed solution belongs to $W^{1,p}(\Omega)$ and satisfies the weak equation. Strict convexity of the gradient energy identifies it with the unique weak solution.

To prove \eqref{eq:background-estimates}, first note that multiplication in $X$ gives $\|\operatorname{div}((A_\gamma-B(\zeta))\nabla\cdot)\|_{Y_0\to Z}\le C\|A_\gamma-B(\zeta)\|_X$. The right side tends to zero as $\gamma\to1$. A Neumann series about $L_0^{-1}$ gives a uniformly bounded inverse for $L_\gamma:=\operatorname{div}(A_\gamma\nabla\cdot):Y_0\to Z$ on a sufficiently small ball. For a weight direction $q\in X$, the derivative $r=D_\gamma u_\gamma[q]$ satisfies $L_\gamma r=-\operatorname{div}(qJ(\nabla u_\gamma))$ and $r|_{\p\Omega}=0$. The inverse bound and multiplication in $X$ give $\|r\|_Y\le C\|q\|_X$. Taking a convex ball and integrating along the segment between $\gamma_1$ and $\gamma_2$ proves the estimate for $u_{\gamma_1}-u_{\gamma_2}$. The estimate for the matrices follows from $A_\gamma=\gamma B(\nabla u_\gamma)$ and smooth composition in $X$.

The flux $\gamma J(\nabla u_{\gamma,f})$ depends smoothly on $(\gamma,f)$ with values in $X^n$. Taking its trace and dot product with $\nu$ is continuous into $C^{k,\alpha}(\p\Omega)$ or $H^{s-1/2}(\p\Omega)$. This proves smoothness of the DN map. For $v_h=D_fu_{\gamma,f_0}[h]$, differentiation of \eqref{eq:forward-map} and of the boundary condition gives
\begin{equation}\label{eq:first-variation}
\begin{cases}
\operatorname{div}(A_\gamma\nabla v_h)=0 & \text{in }\Omega,\\
v_h=h & \text{on }\p\Omega.
\end{cases}
\end{equation}
For a smooth boundary test function $\psi$, choose an extension $H$ independent of $f$. The smooth dependence of the flux allows us to differentiate \eqref{eq:nonlinear-DN}, giving $\langle D_f\Lambda_\gamma(f_0)h,\psi\rangle=\int_\Omega A_\gamma\nabla v_h\cdot\nabla H\,dx$. This is \eqref{eq:linearized-DN}. Density of smooth functions in $H^{1/2}(\p\Omega)$ extends an equality of these derivatives to the weak conductivity DN maps.

For smooth $\gamma$ and $f$, the equation can be written as $\gamma B(\nabla u):D^2u=-\nabla\gamma\cdot J(\nabla u)$. Our solution belongs to $C^{2,\beta}(\ol\Omega)$ for some $\beta>0$. Its principal matrix is uniformly elliptic and belongs to $C^{1,\beta}$, as does the right-hand side. By the boundary Schauder estimate, $u\in C^{3,\beta}(\ol\Omega)$. Repeating this argument proves smoothness up to the boundary. Solution branches in different regularity spaces agree on their common domain by uniqueness of the weak solution. Differentiating this identity also identifies their parameter derivatives.
\end{proof}

For the unit weight, $u_1=\zeta\cdot x$ and $A_1=B(\zeta)$. More generally, if $\zeta\cdot\nabla\gamma_0=0$, then $\operatorname{div}(\gamma_0J(\zeta))=\zeta\cdot\nabla\gamma_0=0$. The reference solution is $u_{\gamma_0}=\zeta\cdot x$, with $A_{\gamma_0}=\gamma_0B(\zeta)$.

\section{Conductivity tensors and coordinate changes}
\label{sec:coordinate-changes}

We use the convention $\Delta_g=\operatorname{div}_g\nabla_g$, so that $-\Delta_g$ is nonnegative. For a smooth Riemannian metric $g$, write $\Lambda_gf=\p_{\nu_g}u_f|_{\p\Omega}$, where $\Delta_gu_f=0$ and $u_f=f$ on $\p\Omega$. The normal $\nu_g$ is the outward unit normal for $g$. We write $\mathsf e$ for the Euclidean metric and $\Diff^\infty(\ol\Omega)$ for the smooth diffeomorphisms of $\ol\Omega$ onto itself.

For $n\ge3$, conductivity matrices and metrics are related by $A_g=(\det g)^{1/2}g^{-1}$ and $g_A=(\det A)^{1/(n-2)}A^{-1}$. The identity $\operatorname{div}(A_g\nabla u)=(\det g)^{1/2}\Delta_gu$ relates their equations. Their DN maps also involve the boundary surface measures, which we compare below.

\begin{lemma}\label{lem:boundary-metric}
Let $n\ge3$, and let $A_1,A_2$ be smooth positive definite symmetric conductivities on $\ol\Omega$. Set $g_j=g_{A_j}$. If $\mathcal D_{A_1}=\mathcal D_{A_2}$, then $g_1$ and $g_2$ induce the same metric on $\p\Omega$, and $\Lambda_{g_1}=\Lambda_{g_2}$.
\end{lemma}

\begin{proof}
Let $h_j$ be the induced boundary metric of $g_j$, and let $h_{\mathsf e}$ be the induced Euclidean metric. Fix $j\in\{1,2\}$. For smooth boundary data $f,\psi$, let $u_f$ be the $A_j$-harmonic function with trace $f$ and let $H$ be a smooth extension of $\psi$. Green's formula gives
\[
\int_{\p\Omega}(\mathcal D_{A_j}f)\psi\,dS_{\mathsf e}
=\int_\Omega A_j\nabla u_f\cdot\nabla H\,dx
=\int_{\p\Omega}(\Lambda_{g_j}f)\psi\,dS_{g_j}.
\]
As operators on smooth functions, we have
\begin{equation}\label{eq:boundary-density}
\mathcal D_{A_j}=w_j\Lambda_{g_j},\quad w_j=\frac{dS_{g_j}}{dS_{\mathsf e}}
=\Big(\frac{\det h_j}{\det h_{\mathsf e}}\Big)^{1/2}.
\end{equation}
The principal symbol of $\Lambda_{g_j}$ in boundary coordinates is $(h_j^{\alpha\beta}\xi_\alpha\xi_\beta)^{1/2}$; see \cite[Proposition~1.3]{LeeUhlmann1989} and \cite[Theorem~1.1(ii)]{JoshiLionheart2005}. The principal symbol of $\mathcal D_{A_j}$ is $w_j$ times this expression. Equality of the conductivity maps gives $w_1^2h_1^{-1}=w_2^2h_2^{-1}$ after squaring the symbols. Substituting the density factors, we obtain
\begin{equation}\label{eq:boundary-symbol-matrices}
(\det h_1)h_1^{-1}=(\det h_2)h_2^{-1}.
\end{equation}
The matrices in \eqref{eq:boundary-symbol-matrices} have size $n-1$. Taking their determinants gives $(\det h_1)^{n-2}=(\det h_2)^{n-2}$. Both determinants are positive and $n-2>0$, so $\det h_1=\det h_2$. Returning to \eqref{eq:boundary-symbol-matrices} gives $h_1=h_2$ and $w_1=w_2$. Division by this common positive function in \eqref{eq:boundary-density} proves $\Lambda_{g_1}=\Lambda_{g_2}$.
\end{proof}

For a diffeomorphism $F$, the conductivity transformation law is
\begin{equation}\label{eq:conductivity-pushforward}
(F_*A)(y)=\frac{DF(x)A(x)DF(x)^T}{|\det DF(x)|},\quad x=F^{-1}(y).
\end{equation}
Changing variables in the weak equation shows that $u\circ F^{-1}$ is $F_*A$-harmonic whenever $u$ is $A$-harmonic. If $F$ fixes the boundary pointwise, the same change of variables gives $\mathcal D_{F_*A}=\mathcal D_A$. For the associated metrics, put $L=DF(x)$. Taking determinants in \eqref{eq:conductivity-pushforward} gives $\det(F_*A)(F(x))=(\det A(x))|\det L|^{2-n}$. The definition of $g_A$ now reads
\begin{equation}\label{eq:metric-covariance}
g_{F_*A}(F(x))=L^{-T}g_A(x)L^{-1}.
\end{equation}
In other words, $g_{F_*A}=F_*g_A=(F^{-1})^*g_A$.

We now remove the diffeomorphism for the linearized conductivities. The proof extends the planar argument in \cite[Section~3]{CarsteaFeizmohammadi2025} to all dimensions $n\ge2$.

\begin{proposition}[Elimination of the diffeomorphism]\label{prop:eliminate-diffeomorphism}
Let $n\ge2$ and $p>1$. Let $\gamma_j\in C^\infty(\ol\Omega)$ be positive, and let $u_j\in C^\infty(\ol\Omega)$ solve \eqref{eq:weighted-p-laplace} for $\gamma_j$ with the same smooth boundary value $f$. Suppose $|\nabla u_j|>0$ on $\ol\Omega$, and set $A_j=\gamma_jB(\nabla u_j)$. If $F\in\Diff^\infty(\ol\Omega)$ fixes $\p\Omega$ pointwise and $F_*A_1=A_2$, then $F=\Id$, $u_1=u_2$ and $\gamma_1=\gamma_2$.
\end{proposition}

\begin{proof}
Since $B(a)a=(p-1)J(a)$, the nonlinear equation for $u_j$ implies $\operatorname{div}(A_j\nabla u_j)=0$. The function $u_1\circ F^{-1}$ is $A_2$-harmonic. Its boundary value is $f$, because $F$ fixes the boundary. Uniqueness for the linear Dirichlet problem gives
\begin{equation}\label{eq:background-transform}
u_2=u_1\circ F^{-1}.
\end{equation}

Fix $x\in\Omega$ and put $y=F(x)$, $L=DF(x)$ and $a_j=\nabla u_j$. The chain rule in \eqref{eq:background-transform} gives $a_1(x)=L^Ta_2(y)$. Set $\ell=|a_2(y)|$, $e=a_2(y)/\ell$, $M=LL^T$ and $t=e^TMe$. These quantities satisfy $\ell>0$, $t>0$, $|a_1(x)|^2=\ell^2t$ and $La_1(x)=\ell Me$. Substituting them into \eqref{eq:conductivity-pushforward} gives
\begin{equation}\label{eq:pushed-linearization}
(F_*A_1)(y)
=\frac{\gamma_1(x)\ell^{p-2}t^{(p-2)/2}}{|\det L|}\Big(M+\frac{p-2}{t}(Me)\otimes(Me)\Big).
\end{equation}
On the other hand, $A_2(y)=\gamma_2(y)\ell^{p-2}(I+(p-2)e\otimes e)$. Comparing this with \eqref{eq:pushed-linearization} and cancelling the positive common factors, we obtain
\begin{equation}\label{eq:matrix-identity}
M+\frac{p-2}{t}(Me)\otimes(Me)=d\big(I+(p-2)e\otimes e\big),
\end{equation}
where $d=\gamma_2(y)|\det L|/(\gamma_1(x)t^{(p-2)/2})>0$. Applying both sides of \eqref{eq:matrix-identity} to $e$ gives $(p-1)Me=d(p-1)e$, since $e^TMe=t$. As $p>1$, we have $Me=de$ and $t=d$. The rank one terms in \eqref{eq:matrix-identity} are now both $(p-2)d\,e\otimes e$. After cancellation, $M=dI$.

It remains to use the boundary condition. We have $DF=\lambda O$, where $\lambda=d^{1/2}>0$ and $O=\lambda^{-1}DF$ is orthogonal. We do not need to assume that $O$ preserves orientation. Since $F=x$ on $\p\Omega$, the divergence theorem gives $\int_\Omega\operatorname{tr}DF\,dx=\int_{\p\Omega}F\cdot\nu\,dS=n|\Omega|$. Also $F$ maps $\Omega$ onto itself, and change of variables gives $\int_\Omega\lambda^n\,dx=\int_\Omega|\det DF|\,dx=|\Omega|$. Using $\operatorname{tr}O\le n$ and H\"older's inequality, we obtain
\begin{equation}\label{eq:conformal-identity}
\begin{aligned}
n|\Omega|=\int_\Omega\operatorname{tr}DF\,dx\le n\int_\Omega\lambda\,dx\le n|\Omega|^{1-1/n}\bigg(\int_\Omega\lambda^n\,dx\bigg)^{1/n}
=n|\Omega|.
\end{aligned}
\end{equation}
Both inequalities in \eqref{eq:conformal-identity} are equalities. The equality condition in H\"older's inequality, applied to $\lambda$ and $1$, makes $\lambda$ constant almost everywhere. Since $\lambda$ is continuous and $\int_\Omega\lambda^n\,dx=|\Omega|$, we have $\lambda=1$ on $\ol\Omega$. The continuous function $n\lambda-\operatorname{tr}DF$ is nonnegative and has zero integral, so $\operatorname{tr}O=n$ at every point. For an orthogonal matrix, $|O-I|_{\mathrm F}^2=2n-2\operatorname{tr}O$, where $|\cdot|_{\mathrm F}$ is the Frobenius norm. We obtain $O=I$ and $DF=I$. Connectedness of $\Omega$ makes $F-x$ constant, and its boundary value is zero. This proves $F=\Id$. Equation~\eqref{eq:background-transform} gives $u_1=u_2$. Now $A_1=A_2$ and positive definiteness of $B(\nabla u_1)$ give $\gamma_1=\gamma_2$.
\end{proof}

\section{Proofs of the rigidity results}
\label{sec:proofs}

We first state the conductivity consequence of the Euclidean metric rigidity theorems. Matrix norms are taken componentwise.

\begin{lemma}\label{lem:constant-conductivity}
Let $n\ge3$ and fix a constant positive definite symmetric matrix $A_*$. There is $\varepsilon>0$ such that, if a smooth positive definite symmetric conductivity $A$ satisfies $\|A-A_*\|_{C^{k,\alpha}(\ol\Omega)}<\varepsilon$ and $\mathcal D_A=\mathcal D_{A_*}$, then $F_*A=A_*$ for some $F\in\Diff^\infty(\ol\Omega)$ fixing $\p\Omega$ pointwise. Here $k,\alpha$ are as in Theorem~\ref{thm:constant-holder}. The same conclusion holds with smallness in $H^s(\Omega)$ for an integer $s>n/2+1$.
\end{lemma}

\begin{proof}
Let $d_*=\det A_*$, set $S=d_*^{1/(2(n-2))}A_*^{-1/2}$, and put $\Omega'=S\Omega$. The map $S$ is fixed and positive definite. We have $\det S=d_*^{1/(n-2)}$ and $SA_*S^T=d_*^{1/(n-2)}I$. The transformed conductivity $\widehat A(y)=SA(S^{-1}y)S^T/\det S$ has constant reference $\widehat A_*=I$.

For smooth boundary functions $f',\psi'$ on $\p\Omega'$, let $v$ be $A$-harmonic with trace $f'\circ S$, and put $v'=v\circ S^{-1}$. If $H'$ extends $\psi'$ and $H=H'\circ S$, the chain rule and $dy=(\det S)\,dx$ give
\[
\int_{\Omega'}\widehat A\nabla v'\cdot\nabla H'\,dy
=\int_\Omega A\nabla v\cdot\nabla H\,dx.
\]
This proves that $v'$ is $\widehat A$-harmonic and that $\langle\mathcal D_{\widehat A}^{\Omega'}f',\psi'\rangle=\langle\mathcal D_A^\Omega(f'\circ S),\psi'\circ S\rangle$. Applying this identity to both $A$ and $A_*$ proves $\mathcal D_{\widehat A}^{\Omega'}=\mathcal D_I^{\Omega'}$.

Write $X=C^{k,\alpha}$ or $X=H^s$, according to the smallness condition. The fixed linear change gives $\|\widehat A-I\|_{X(\Omega')}\le C\|A-A_*\|_{X(\Omega)}$. The matrix function $\mathcal C(P)=(\det P)^{1/(n-2)}P^{-1}$ is smooth near $I$. Using $\mathcal C(I+H)-\mathcal C(I)=\int_0^1D\mathcal C(I+tH)[H]\,dt$, smooth composition and multiplication in $X$, we obtain
\begin{equation}\label{eq:constant-metric-estimate}
\|g-\mathsf e\|_{X(\Omega')}\le C\|A-A_*\|_{X(\Omega)},\quad g=g_{\widehat A}.
\end{equation}
The constant is uniform for $A$ in a sufficiently small fixed ball.

Lemma~\ref{lem:boundary-metric} gives $g|_{T\p\Omega'}=\mathsf e|_{T\p\Omega'}$ and $\Lambda_g=\Lambda_{\mathsf e}$. For sufficiently small $\varepsilon$, estimate \eqref{eq:constant-metric-estimate} allows us to apply \cite[Theorem~1.1]{Lin2026NearEuclidean} in the H\"older case or \cite[Theorem~1]{Stefanov2026NearEuclidean} in the Sobolev case. We get $g=Y^*\mathsf e$ for a smooth diffeomorphism $Y$ fixing $\p\Omega'$. By \eqref{eq:metric-covariance}, $g_{Y_*\widehat A}=Y_*g=\mathsf e$, and the conductivity--metric correspondence gives $Y_*\widehat A=I$. Set $F=S^{-1}\circ Y\circ S$. The conductivity transformation law shows that $F_*A=A_*$. The boundary condition for $Y$ gives $F|_{\p\Omega}=\Id$.
\end{proof}

\begin{proof}[Proof of Theorem~\ref{thm:constant-holder} and Corollary~\ref{cor:constant-sobolev}]
By Lemma~\ref{lem:noncritical-linearization}, the solution $u_\gamma$ with trace $f_0$ has no critical points. The reference solution is $u_1=\zeta\cdot x$, and the corresponding matrices are $A_\gamma=\gamma B(\nabla u_\gamma)$ and $A_1=B(\zeta)$. Estimate \eqref{eq:background-estimates} gives $\|A_\gamma-A_1\|_X\le C\|\gamma-1\|_X$ in the chosen H\"older or Sobolev norm. The hypothesis \eqref{eq:constant-linearized-data} and \eqref{eq:linearized-DN} give $\mathcal D_{A_\gamma}=\mathcal D_{A_1}$. Choose $\varepsilon$ small enough for Lemma~\ref{lem:constant-conductivity}. We obtain a diffeomorphism $F$ fixing $\p\Omega$ with $F_*A_\gamma=A_1$. Proposition~\ref{prop:eliminate-diffeomorphism}, applied to $u_\gamma$ and $u_1$, gives $\gamma=1$.
\end{proof}

For the nonconstant reference, we use \cite[Theorem~2]{Stefanov2026NearEuclidean}. On a fixed smooth compact domain in dimension $n\ge3$, let $s,m$ be integers with $s>n/2+1$, $m\ge s+2$ and $m>3n/2+5$. There are $\delta_*,\varepsilon_*>0$, uniform over smooth reference metrics $g_0=e^{2\varphi}\mathsf e$ with $\|\varphi\|_{C^m}<\delta_*$, for which the following statement holds. If $g$ is smooth, $\|g-g_0\|_{H^s}<\varepsilon_*$, and $g,g_0$ have the same induced boundary metric and geometric DN map, then $g=Y^*g_0$ for a smooth diffeomorphism $Y$ fixing the boundary.

\begin{proof}[Proof of Theorem~\ref{thm:directional-reference}]
Since $|\zeta|=1$ and $\zeta\cdot\nabla\gamma_0=0$, we have $\operatorname{div}(\gamma_0J(\zeta))=0$. The reference solution is $u_{\gamma_0}=\zeta\cdot x$, and $A_{\gamma_0}=\gamma_0B(\zeta)$. We first construct solutions for weights near this fixed $\gamma_0$. The construction is centered at $\gamma_0$, which is not assumed to be close to $1$.
	
	Set $X=H^s(\Omega)$, $Y=H^{s+1}(\Omega)$, $Y_0=Y\cap H_0^1(\Omega)$ and $\mathcal B=H^{s+1/2}(\p\Omega)$. Choose a bounded right inverse $E:\mathcal B\to Y$ of the trace map. Use the map $\mathcal F$ in \eqref{eq:forward-map} with $f=f_0+h$ and $u=\zeta\cdot x+Eh+w$, now near $(\gamma,h,w)=(\gamma_0,0,0)$. We restrict $\gamma$ to a positive neighborhood of $\gamma_0$ and $(h,w)$ so that $\|\nabla Eh+\nabla w\|_{L^\infty}<1/2$. Multiplication and smooth composition in $H^s$ make $\mathcal F$ a smooth map from this neighborhood in $X\times\mathcal B\times Y_0$ to $H^{s-1}(\Omega)$. At its center, $\mathcal F(\gamma_0,0,0)=0$ and
	\[
	D_w\mathcal F(\gamma_0,0,0)[v]=\operatorname{div}(\gamma_0B(\zeta)\nabla v)=:L_{\gamma_0}v.
	\]
	The coefficient matrix is smooth and satisfies $\xi^T\gamma_0(x)B(\zeta)\xi\ge(\min_{\ol\Omega}\gamma_0)\min\{1,p-1\}|\xi|^2$ for every $x\in\ol\Omega$ and $\xi\in\mathbb R^n$. Coercivity gives unique solvability of the weak Dirichlet problem. Boundary elliptic regularity gives $v\in H^{s+1}(\Omega)$ for a right-hand side in $H^{s-1}(\Omega)$, together with $\|v\|_{H^{s+1}}\le C_{\gamma_0}\|L_{\gamma_0}v\|_{H^{s-1}}$ for $v\in Y_0$. In particular, $L_{\gamma_0}:Y_0\to H^{s-1}(\Omega)$ is an isomorphism.

The implicit function theorem gives a smooth branch $w=w(\gamma,h)$ with $w(\gamma_0,0)=0$. The embedding $Y\hookrightarrow C^1(\ol\Omega)$ allows us to choose a neighborhood of $\gamma_0$ and a neighborhood of $0$ in $\mathcal B$ on whose product $\|\nabla Eh+\nabla w(\gamma,h)\|_{L^\infty}<1/2$. The solution $u_{\gamma,f_0+h}=\zeta\cdot x+Eh+w(\gamma,h)$ has the prescribed boundary value and satisfies $|\nabla u_{\gamma,f_0+h}|>1/2$ on $\ol\Omega$. Strict convexity of the gradient energy identifies it with the unique weak solution. For smooth weights and boundary data, the boundary regularity argument at the end of the proof of Lemma~\ref{lem:noncritical-linearization} gives smoothness up to the boundary.

We prove a uniform estimate on a neighborhood of this fixed reference weight. Write $u_\gamma=u_{\gamma,f_0}$, $A_\gamma=\gamma B(\nabla u_\gamma)$ and $L_\gamma=\operatorname{div}(A_\gamma\nabla\cdot)$. Continuity of the solution branch and smooth composition give $A_\gamma\to A_{\gamma_0}$ in $H^s$ as $\gamma\to\gamma_0$ in $H^s$. For $v\in Y_0$, multiplication in $H^s$ and boundedness of differentiation give
	\[
	\|(L_\gamma-L_{\gamma_0})v\|_{H^{s-1}}\le C\|A_\gamma-A_{\gamma_0}\|_{H^s}\|v\|_{H^{s+1}}.
	\]
	Choose $r_0>0$ so that the construction above applies and $\|L_{\gamma_0}^{-1}(L_\gamma-L_{\gamma_0})\|_{Y_0\to Y_0}<1/2$ whenever $\|\gamma-\gamma_0\|_{H^s}<r_0$. A Neumann series gives a uniform bound for $L_\gamma^{-1}:H^{s-1}(\Omega)\to Y_0$ on this neighborhood. If $q\in H^s(\Omega)$ and $r_q=D_\gamma u_\gamma[q]$, differentiation of the equation gives $L_\gamma r_q=-\operatorname{div}(qJ(\nabla u_\gamma))$ and $r_q|_{\p\Omega}=0$. The inverse bound and multiplication in $H^s$ imply $\|r_q\|_{H^{s+1}}\le C_{\gamma_0}\|q\|_{H^s}$. Integrating along $\gamma_t=\gamma_0+t(\gamma-\gamma_0)$, $0\le t\le1$, gives $\|u_\gamma-\zeta\cdot x\|_{H^{s+1}}\le C_{\gamma_0}\|\gamma-\gamma_0\|_{H^s}$. Each $\gamma_t$ belongs to the same neighborhood. Using $A_\gamma-A_{\gamma_0}=(\gamma-\gamma_0)B(\nabla u_\gamma)+\gamma_0(B(\nabla u_\gamma)-B(\zeta))$ and smooth composition in $H^s$, we obtain
	\[
	\|u_\gamma-\zeta\cdot x\|_{H^{s+1}(\Omega)}+\|A_\gamma-A_{\gamma_0}\|_{H^s(\Omega)}\le C_{\gamma_0}\|\gamma-\gamma_0\|_{H^s(\Omega)}.
	\]
	The constants may depend on $\gamma_0$, but they are independent of $\gamma$ with $\|\gamma-\gamma_0\|_{H^s}<r_0$.

The flux $\gamma J(\nabla u_{\gamma,f})$ depends smoothly on $(\gamma,f)$ with values in $H^s(\Omega)^n$, and its normal trace is a smooth $H^{s-1/2}(\p\Omega)$-valued map. For a smooth boundary function $h$, the derivative $v_h=D_fu_{\gamma,f_0}[h]$ satisfies $\operatorname{div}(A_\gamma\nabla v_h)=0$ and $v_h|_{\p\Omega}=h$. If $\psi$ is a smooth boundary test function and $H$ is an extension independent of $f$, differentiating the weak flux pairing gives $\langle D_f\Lambda_\gamma(f_0)[h],\psi\rangle=\int_\Omega A_\gamma\nabla v_h\cdot\nabla H\,dx$. This proves $D_f\Lambda_\gamma(f_0)=\mathcal D_{A_\gamma}$ for every weight in the chosen neighborhood, including $\gamma_0$. Equality on smooth boundary functions extends to the weak conductivity DN maps by density.
	
	Put $B_\zeta=B(\zeta)$, fix $S=(\det B_\zeta)^{1/(2(n-2))}B_\zeta^{-1/2}$, and set $\Omega'=S\Omega$. For $\eta=\gamma_0$ or $\eta=\gamma$, write $\widehat A_\eta(y)=SA_\eta(S^{-1}y)S^T/\det S$. The map $S$ depends only on $n,p,\zeta$. The calculation in Lemma~\ref{lem:constant-conductivity} gives $SB_\zeta S^T/\det S=I$. Since $A_{\gamma_0}=\gamma_0B_\zeta$, we have $\widehat A_{\gamma_0}(y)=\gamma_0(S^{-1}y)I$. Its associated metric is
	\begin{equation}\label{eq:conformal-reference}
		g_0=g_{\widehat A_{\gamma_0}}=e^{2\varphi}\mathsf e,\quad \varphi(y)=\frac{1}{n-2}\log\gamma_0(S^{-1}y).
	\end{equation}
	This is a fixed smooth conformally Euclidean metric. There is no restriction on the size of $\varphi$.

	Set $g=g_{\widehat A_\gamma}$. The fixed linear change of variables gives $\|\widehat A_\gamma-\widehat A_{\gamma_0}\|_{H^s(\Omega')}\le C_{\gamma_0}\|\gamma-\gamma_0\|_{H^s(\Omega)}$. After decreasing $r_0$, both transformed conductivities have eigenvalues in a fixed compact subinterval of $(0,\infty)$ and have uniformly bounded $H^s$ norms. These bounds also hold along the segment $P_t=\widehat A_{\gamma_0}+t(\widehat A_\gamma-\widehat A_{\gamma_0})$. For $\mathcal C(P)=(\det P)^{1/(n-2)}P^{-1}$, the identity $\mathcal C(\widehat A_\gamma)-\mathcal C(\widehat A_{\gamma_0})=\int_0^1D\mathcal C(P_t)[\widehat A_\gamma-\widehat A_{\gamma_0}]\,dt$, together with smooth composition and multiplication in $H^s$, gives
	\begin{equation}\label{eq:reference-metric-estimate}
		\|g-g_0\|_{H^s(\Omega')}\le C_{\gamma_0}\|\gamma-\gamma_0\|_{H^s(\Omega)}.
	\end{equation}
	Let $\varepsilon_*>0$ be the threshold in \cite[Theorem~2]{Stefanov2026NearEuclidean} for the fixed metric $g_0$ on $\ol\Omega'$ and the chosen integer $s$. Choose $0<\varepsilon<r_0$ so that $C_{\gamma_0}\varepsilon<\varepsilon_*$, with $C_{\gamma_0}$ taken from \eqref{eq:reference-metric-estimate}.
	
	The data assumption and the linearization proved above give $\mathcal D_{A_\gamma}=\mathcal D_{A_{\gamma_0}}$. The change of variables in the weak pairing transfers this equality to $\Omega'$, as in Lemma~\ref{lem:constant-conductivity}. Lemma~\ref{lem:boundary-metric} gives equality of the induced boundary metrics and $\Lambda_g=\Lambda_{g_0}$. Applying the conformal metric rigidity theorem with \eqref{eq:conformal-reference} and \eqref{eq:reference-metric-estimate}, we obtain $g=Y^*g_0$ for a smooth diffeomorphism $Y$ fixing $\p\Omega'$. Equation~\eqref{eq:metric-covariance} gives $Y_*\widehat A_\gamma=\widehat A_{\gamma_0}$. The map $F=S^{-1}\circ Y\circ S$ fixes $\p\Omega$ and satisfies $F_*A_\gamma=A_{\gamma_0}$. Both background solutions are smooth, have nonvanishing gradients, and have the same boundary value $f_0$. Proposition~\ref{prop:eliminate-diffeomorphism} gives $\gamma=\gamma_0$.

\end{proof}

\begin{remark}\label{rem:reference-class}
Proposition~\ref{prop:eliminate-diffeomorphism} applies to both weights in the same way. The restriction on the reference weight comes from the metric rigidity theorem used to obtain the diffeomorphism. The inverse data in our results involve only the first boundary derivative, with no derivative with respect to the unknown weight.
\end{remark}

\section*{Statements and declarations}

\para{Data availability statement} No datasets were generated or analyzed in this study.

\para{Conflict of interest} The author declares no conflict of interest.

\para{Acknowledgments} The author is partially supported by the National Science and Technology Council (NSTC), Taiwan, under the project 113-2115-M-A49-017-MY3. The author also acknowledges financial support from the Alexander von Humboldt Foundation through the Henriette Herz Scouting Programme, hosted by Universit\"at Duisburg-Essen, Germany.The density

\bibliographystyle{alpha}
\bibliography{refs}

\end{document}